\documentclass{conm-p-l}
\usepackage{amssymb, graphicx, labelfig
}

\newtheorem{thm}{Theorem}
\newtheorem{prop}[thm]{Proposition}
\newtheorem{lem}[thm]{Lemma}

\newtheorem{fact}[thm]{Fact}
\newtheorem{conj}[thm]{Conjecture}

\theoremstyle{remark}

\theoremstyle{definition}
\newtheorem{defi}[thm]{Definition}

\newcommand{\col}{\kern -3pt :}
\newcommand{\C}{\mathbb C}
\newcommand{\R}{\mathbb R}
\newcommand{\Z}{\mathbb Z}

\newcommand{\Id}{\mathrm{Id}}
\newcommand{\End}{\mathrm{End}}

\newcommand{\Spin}{\mathrm{Spin}}
\newcommand{\RS}{\mathcal R_{\PSL(\C)}^{\Spin}}

\newcommand{\T}{\mathcal T}
\newcommand{\SSS}{\mathcal S^A}

\newcommand{\RR}{\mathcal R}

\newcommand{\SL}{\mathrm{SL}_2}
\newcommand{\PSL}{\mathrm{PSL}_2}
\newcommand{\Tr}{\mathrm{Tr}}

\newcommand{\db}{/\kern -4pt/}

\renewcommand{\leq}{\leqslant}

\renewcommand{\phi}{\varphi}
\renewcommand{\epsilon}{\varepsilon}

\title[Kauffman brackets, characters and triangulations]
{Kauffman brackets, character varieties\\and triangulations of surfaces}

\author{Francis Bonahon}
\address {Department
of Mathematics,  University of
Southern California, Los Angeles
CA~90089-2532, U.S.A.}
\email{fbonahon@math.usc.edu}

\author{Helen Wong}
\address{Department
of Mathematics, Carleton College, Northfield MN 55057, U.S.A.}
\email{hwong@carleton.edu}

\thanks{This research was partially supported by the grant  DMS-0604866 from the National Science Foundation, and by a mentoring grant from the Association for Women in Mathematics.}
\date{\today}

\begin{document}
\maketitle

\begin{abstract}
A Kauffman bracket on a surface is an invariant for framed links in the thickened surface, satisfying the Kauffman skein relation and multiplicative under superposition. This includes representations of the skein algebra of the surface. We show how an irreducible representation of the skein algebra usually specifies a point of the character variety of homomorphisms from the fundamental group of the surface to $\PSL(\C)$, as well as  certain weights associated to the punctures of the surface. Conversely, we sketch a proof of the fact that each point of the character variety, endowed with appropriate puncture weights, uniquely determines a Kauffman bracket. Details will appear elsewhere. 
\end{abstract}

This article closely follows the talk given by one of us at the  conference \textit{Topology and geometry in dimension three: triangulations, invariants and geometric structures}. It turns out to feature all three of the scientific themes emphasized in the title of the conference. The main object under study consists of Kauffman brackets, which are invariants of links in 3--manifolds. We will connect these to hyperbolic metrics on a thickened surface, and finally use triangulations of surfaces as a tool in our constructions. The article describes a long-term program, and focusses more on statements than proofs. Details will appear elsewhere \cite{BonWon1, BonWon2}. 

The two authors were greatly honored to be invited to this great celebration of 3--dimensional topology and of the achievements of one of its major contributors, and greatly enjoyed their participation. They are very pleased to dedicate this article to Bus Jaco, as a grateful acknowledgment of his impact on their field. 

\section{Kauffman brackets and the skein algebra}

\subsection{The classical Kauffman bracket}
For a fixed  number $A\in \C^*$, 
the classical {\emph{Kauffman bracket}} is the unique map
$$
\mathcal K \col \{ \text{framed links } K \subset S^3 \} \to \C
$$ 
that satisfies the following four properties.

\begin{enumerate}

\item \textsc{Isotopy Invariance:} If the framed links $K_1$ and $K_2$ are isotopic, then $\mathcal K(K_1)= \mathcal K(K_2)$.

\item \textsc{Skein Relation:} Let the three framed links $K_1$, $K_0$ and $K_\infty\subset S^3$ form a \emph{Kauffman triple}, in the sense that  the only place where they differ is in a small ball of $S^3$, where they are as represented in Figure~\ref{fig:skein} and where the framings are all pointing towards the reader. Then
$$
 \mathcal K(K_1) = A^{-1}\mathcal K(K_0) + A \mathcal K(K_\infty).
$$

\item\textsc{Superposition Relation:} Let $K_1\cdot K_2$ be the framed link obtained by stacking $K_2$ on top of  $K_1$, as in Figure~\ref{fig:Superposition}. Then
$$\mathcal K (K_1\cdot  K_2) = \mathcal K(K_1) \, \mathcal K(K_2).$$

\item \textsc{Non-triviality:} There exists a non-empty link $K$ such that $\mathcal K(K) \not=0$.

\end{enumerate}

\begin{figure}[htbp]
\SetLabels
( .5 * -.4 ) $K_0$ \\
( .1 * -.4 )  $K_1$\\
(  .9*  -.4) $K_\infty$ \\
\endSetLabels
\centerline{\AffixLabels{\includegraphics{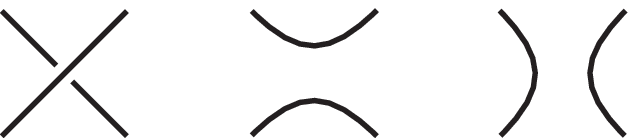}}}
\vskip 15pt
\caption{A Kauffman triple}
\label{fig:skein}
\end{figure}

\begin{figure}[htbp]
{\SetLabels
(.14 * -.2) $K_1 $ \\
(.51 * -.2) $ K_2$ \\
( .88*-.2 ) $ K_1\cdot K_2$ \\
\endSetLabels
\centerline{\AffixLabels{ \includegraphics{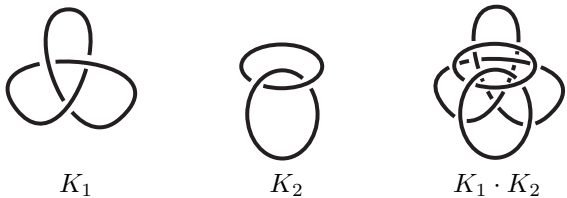}}}}
\vskip 10pt
\caption{Superposing two links}
\label{fig:Superposition}
\end{figure}

Recall that a \emph{framing} for a link $K$ in $S^3$, or more generally in any 3--dimensional manifold, is the choice at each $x\in K$ of a vector $v_x$ that depends differentiably on $x$ and is never tangent to $K$.

L.~Kauffman \cite{Kau} proved the following fundamental result.

\begin{thm}
\label{thm:Kauffman}
For every $A\in \C^*$, there exists a unique  Kauffman bracket
$$
\mathcal K \col \{ \text{framed links } K \subset S^3 \} \to \C.
$$
\vskip -\belowdisplayskip
\vskip -\baselineskip
 \qed

\end{thm}

In addition, it easily follows from the skein relation and the invariance under isotopy  that the Kauffman bracket $\mathcal K(K)$ of a framed link $K$  is a Laurent polynomial in $A$ with integer coefficients. 

\subsection{Kauffman brackets on a surface} 
The theory of knots and links in $S^3$ can be described from a purely 2--dimensional point of view, by considering their projections to the plane up to Reidemeister moves. In this regard, the existence and uniqueness of the Kauffman bracket can be considered as a property of the plane, or of the disk.
It is then tempting  to extend this property to more general surfaces.

More precisely, let $S$ be a compact oriented surface, possibly with boundary. We can then consider ``pictures of knots and links'' on the surface $S$, considered up to the Reidemeister moves II and III. Such a link diagram in $S$ determines an isotopy class of framed links in the thickened surface $S\times [0,1]$, where the framing is vertical pointing upwards, namely  parallel to the $[0,1]$ factor and pointing in the direction of $1$. (Reidemeister move I is excluded as it would alter the framing.) 

The Skein Relation and the Superposition Relation have automatic extensions to this context of framed links in the thickened surface $S\times[0,1]$. More precisely, a \emph{Kauffman triple} in $S\times[0,1]$ consists of three framed links $K_1$, $K_0$, $K_\infty$ that differ only above a small disk in $S$, where they are represented by the link diagrams of  Figure~\ref{fig:skein} (with vertical framing pointing upwards). 

Also, if $K_1$ and $K_2$ are two framed links in $S\times [0,1]$, we can consider the framed links $K_1'\subset S\times [0,\frac12]$ and $K_2' \subset S \times[\frac12, 1]$ respectively obtained from $K_1$ and $K_2$ by rescaling in the $[0,1]$ direction. Then  the \emph{superposition} of $K_2$ above $K_1$ is the framed link $K_1\cdot K_2 = K_1' \cup K_2'$. This is illustrated in Figure~\ref{fig:Superposition4} in the case where the surface $S$ is a twice punctured torus. 
Note that, in contrast to the case of framed links in $S^3$, the superposition $K_2\cdot K_1$ is not always isotopic to $K_1 \cdot K_2$ in $S\times[0,1]$.

\begin{figure}[htbp]

\SetLabels
( .16*-.3 ) $ K_1$ \\
( .5* -.3) $ K_2$ \\
( .88*-.3 ) $ K_1\cdot K_2$ \\
\endSetLabels
\centerline{\AffixLabels{ \includegraphics{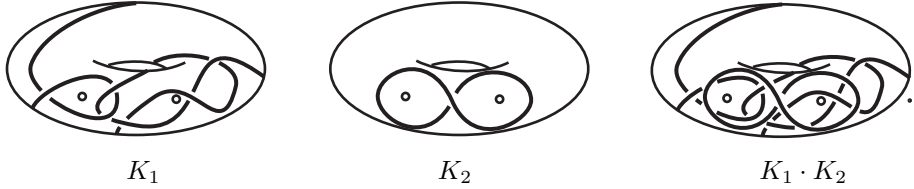}}}
\vskip 15pt
\caption{Superposition in $S\times[0,1]$}
\label{fig:Superposition4}
\end{figure}

It is now natural to introduce the following definition.

\begin{defi}
For a fixed $A\in \C^*$, a  \emph{Kauffman bracket} on the surface $S$, valued in the algebra $\mathcal A$, is a map
$$
\mathcal K \col \{ \text{framed links } K \subset S \times [0,1]\} \to \mathcal A
$$ 
satisfying the following four properties.

\begin{enumerate}

\item \textsc{Isotopy Invariance:} If the framed links $K_1$ and $K_2$ are isotopic in $S\times[0,1]$, then $\mathcal K(K_1)= \mathcal K(K_2)$.

\item \textsc{Skein Relation:} If the three framed links $K_1$, $K_0$ and $K_\infty\subset S^3$ form a Kauffman triple in $S\times [0,1]$, then
$$
 \mathcal K(K_1) = A^{-1}\mathcal K(K_0) + A \mathcal K(K_\infty).
$$

\item \textsc{Superposition Relation:} If $K_1\cdot K_2$ is the framed link obtained by stacking $K_2$ above $K_1$ in $S\times [0,1]$, then 
$$\mathcal K (K_1\cdot  K_2) = \mathcal K(K_1) \, \mathcal K(K_2).$$

\item \textsc{Non-triviality:} There exists a non-empty framed link $K\subset S \times [0,1] $ such that $\mathcal K(K) \not=0$.  

\end{enumerate}
\end{defi}

In view of Theorem~\ref{thm:Kauffman}, it is now natural to aim for the following

\theoremstyle{theorem}
\newtheorem*{MainGoal}{Main Goal}
\begin{MainGoal}
Classify all possible Kauffman brackets for the surface $S$.
\end{MainGoal}

This goal will turn out to be too optimistic at this point, but at least we will construct many examples of interesting Kauffman brackets.

\subsection{The skein algebra}

We first rephrase the definition of Kauffman brackets, following \cite{Prz, Tur1, Bull1, BFK1, PrzS, Tur2, BFK3}. 

\begin{defi}
The \emph{skein algebra} $\SSS(S)$ of the surface $S$ is defined by considering the vector space freely generated (over $\C$) by the set of all framed links in $S\times [0,1]$, and then by taking the quotient of this vector space by the subspace generated by elements of the following two types:
\begin{itemize}
\item[] $K_1-K_2$ for every pair of isotopic framed links $K_1$, $K_2\subset S\times [0,1]$;
\item[] $K_1 - A^{-1} K_0 - A K_\infty$ for every Kauffman triple $K_1$, $K_0$, $K_\infty \subset S\times[0,1]$. 
\end{itemize}
The multiplication of the algebra $\SSS(S)$ is  induced by the superposition operation, with  $[K_1][K_2] = [K_1\cdot K_2]$ for any two  framed links $K_1$, $K_2\subset S\times [0,1]$.
\end{defi}

A \emph{skein} is a class $[K] \in \SSS(S)$ represented by a framed link $K\subset S \times [0,1]$. 
Note that the empty skein $[\varnothing] $ is a unit in the skein algebra $\SSS(S)$. 

The definition of the skein algebra is specially designed so that the following holds.

\begin{fact}
A Kauffman bracket valued in the algebra $\mathcal A$ is the same thing as an algebra homomorphism
$\mathcal K \col \SSS(S) \to \mathcal A  $ that is \emph{non-trivial}, in the sense that there exists a non-empty skein $[K]\in \SSS(S)$ with $\mathcal K([K]) \not=0$.  \qed
\end{fact}

Note that, when $\mathcal A$ contains an idempotent element  $\iota \in \mathcal A$ (for instance a unit element),   the skein algebra always admits a \emph{trivial} algebra  homomorphism $\SSS(S) \to \mathcal A$ sending the unit $[\varnothing]$ to $\iota$ and every  other skein $[K]\not= [\varnothing] \in \SSS(S)$ to~$0$.

Our Main Goal now  becomes:

\begin{MainGoal}
Classify all algebra homomorphisms $\mathcal K \col \SSS(S) \to \mathcal A  $.
\end{MainGoal}

One problem appears right away: there really exists a great variety of algebras $\mathcal A$. For instance, we could take $\mathcal A$ equal to the skein algebra $\SSS(S)$ and $\mathcal K$ equal to the identity map (and call it the universal Kauffman bracket!), but this would clearly have little impact on our understanding of links in $S\times [0,1]$. 

It therefore makes sense to restrict attention to a specific class of algebras, namely to algebras $\mathcal A = \End(V)$ of linear endomorphisms of a finite dimensional vector space $V$. Choosing a basis for $V$, such an algebra is just an algebra of square matrices of a given size, and consequently very explicit. 

Thus, we henceforth limit our investigation to algebra homomorphisms $\mathcal K \col \SSS(S) \to \End(V)$, where $V$ is a finite dimensional vector space. Such an algebra homomorphism is a \emph{representation} of the algebra $\SSS(S)$, and we will reflect this change in emphasis by switching from the letter $\mathcal K$ (for ``Kauffman bracket'') to $\rho$ (for ``representation'') in the notation. We will also restrict our investigation to \emph{irreducible} representations, as building blocks of more general representations. 

\theoremstyle{theorem}
\newtheorem*{RevisedGoal}{Revised Goal}
\begin{RevisedGoal}
Classify all irreducible representations $\rho\col \SSS(S) \to \End(V)$ of the skein algebra $\SSS(S)$. 
\end{RevisedGoal}

This goal now becomes much more realistic. Although we are not able to completely attain it at this point, we will make progress in this direction. We begin by discussing in \S\ref{sect:A=pm1} and \S\ref{sect:SmallSurfaces} a few simple cases where the program has been completed. In \S\S \ref{sect:Invariants}--\ref{sect:ConjClass} we describe a conjectural classification of irreducible representations  $\rho\col \SSS(S) \to \End(V)$, and state an existence result in Theorem~\ref{thm:MainThm}. We conclude with a few indications on its proof in \S \ref{sect:Triang}. 

\section{The case where $A=\pm1$}
\label{sect:A=pm1}

Surprisingly enough, when $A=\pm1$, Kauffman brackets are related to group homomorphisms from the fundamental group $\pi_1(S)$ to $\SL(\C)$ or $\PSL(\C)$. 

It is an immediate consequence of the Skein Relation that the skein algebra $\mathcal S^{\pm1}(S)$ is commutative. Elementary linear algebra then shows that every representation $\mathcal S^{\pm1}(S) \to \End(V)$ splits as a direct sum of irreducible representations, and that every irreducible representation has dimension~1. 
As a consequence, every representation is isomorphic to a direct sum of representations $\rho \col \mathcal S^{\pm1}(S) \to \End(\C) = \C$.

\subsection{The case where $A=-1$}
\label{sect:A=-1}
Consider a group homomorphism $r\col \pi_1(S) \to \SL(\C)$. A closed curve $K$ in $S \times [0,1]$ determines a conjugacy class in $\pi_1(S)$, so that the element $r(K) \in \SL(\C)$ is well-defined up to conjugation. In particular, the trace $\Tr\,r(K) \in \C$ is uniquely determined. 

The corresponding trace map $K \mapsto \Tr\,r(K)$  depends only on the class of $r$ in the \emph{character variety} 
$$
\RR_{\SL(\C)} = \{\text{group homomorphisms } r \col \pi_1(S) \to \SL(\C) \}\db \SL(\C)
$$
where $\SL(\C)$ acts on such group homomorphisms by conjugation. Here the double bar indicates that the quotient has to be taken in the sense of geometric invariant theory \cite{Mum}. In practice, because the trace function $\Tr$ generates all conjugation invariant functions on $\SL(\C)$, two group homomorphisms $r$, $r'\col \pi_1(S) \to \SL(\C)$ represent the same element of the character variety $\RR_{\SL(\R)}$ if and only if $\Tr \,r'(K) = \Tr \,r(K)$ for every closed curve $K\subset S\times [0,1]$. In the generic case where $r$ is irreducible, namely where its image in $\SL(\C)$ leaves no line in $\C^2$ invariant, this is equivalent to the property that $r$ and $r'$ are conjugate by an element of $\SL(\C)$. See for instance \cite{CulS}.

The trace function provides a one-to-one correspondence between Kauffman brackets $\mathcal S^{-1}(S) \to \C$ and elements of the character variety $\RR_{\SL(\C)}$. 

\begin{thm}[\cite{Bull1, PrzS, Vogt, Frick, FrickKlein, Helling, BruHil, Gold2, Luo}]
\label{thm:A=-1}
A group homomorphism $r\col \pi_1(S) \to \SL(\C)$ defines a non-trivial homomorphism
$$
\rho_r\col  \mathcal S^{-1}(S)  \to \C 
$$
by the property that
$$
\rho_r([K]) = - \Tr\,r(K)
$$
for every framed knot $K\subset S\times[0,1]$.

Conversely, every  non-trivial homomorphism $ \mathcal S^{-1}(\Sigma) \to \C$ is associated in this way to a homomorphism $r\col \pi_1(S) \to \SL(\C)$, and the class of $r$ in $\RR_{\SL(\C)}$ is unique.
\end{thm}

The first half of the statement is a crucial observation of D. Bullock \cite{Bull1} that, if we associate to each link $K\subset S\times[0,1]$ with components $K_1$, \dots, $K_n$ the number
$$
\mathcal K_r(K) = (-1)^n \prod_{i=1}^n \Tr\,r(K_i) \in \C,
$$
then $\mathcal K_r$ satisfies the skein relation with $A=-1$. This property is an easy consequence of the classical trace formula in $\SL$, which states that
$$
\Tr(MN) + \Tr(MN^{-1}) = \Tr(M) \Tr(N)
$$
for every $M$, $N\in \SL(R)$
(for an arbitrary commutative unit ring $R$). 

The converse statement, which can be traced back to Vogt and Fricke \cite{Vogt, Frick}, amounts to characterizing which functions on a group $G$ can be realized as the trace function of a homomorphism $G\to \SL(\C)$. See also \cite{Helling, BruHil, PrzS, Luo, Gold2}.

\subsection{The case where $A=+1$}
\label{sect:A=+1}

One can go from $A=-1$ to $A=+1$ by a construction of J.~Barrett \cite{Barr}, which associates an isomorphism $\SSS(S) \to \mathcal S^{-A}(S)$ to each spin structure on $S$. See \cite[\S 2]{PrzS} for a proof that Barrett's linear isomorphism is in fact an algebra isomorphism. 

Let $\Spin(S)$ denote the space of isotopy classes of spin structures on $S$. Recall that the difference between two elements of $\Spin(S)$ is measured by an obstruction in $H^1(S;\Z_2)$. As a consequence, $\Spin(S)\cong H^1(S;\Z_2)$ once we have chosen a base spin structure, and $H^1(S;\Z_2)$ acts freely and transitively on $\Spin(S)$. 

The cohomology group $H^1(S;\Z_2)$ also acts on $\RR_{\SL(\C)}(S)$ by the property that, if $\alpha \in H^1(S;\Z_2)$ and $r\in \RR_{\SL(\C)}(S)$, 
$$
\alpha r(\gamma) = (-1)^{\alpha(\gamma)} r(\gamma) \in \SL(\C)
$$
for every $\gamma \in \pi_1(S)$. 
The quotient space $\RR_{\PSL(\C)}/H^1(S;\Z_2)$ is the subset $\RR_{\PSL(\C)}^0$ of 
$$
\RR_{\PSL(\C)} = \{\text{group homomorphisms } r \col \pi_1(S) \to \PSL(\C) \}\db \PSL(\C)
$$
consisting of those homomorphisms $r \col \pi_1(S) \to \PSL(\C) $ that admit a lift $\pi_1(S) \to \SL(\C) $. Note that $\RR_{\PSL(\C)}^0$ is equal to $\RR_{\PSL(\C)}$ when $\partial S \not =\varnothing$, and is one of the two components of $\RR_{\PSL(\C)}$ when $S$ is closed \cite{Gold1}. 

We now consider the twisted product
\begin{align*}
\RS(S) &= \RR_{\PSL(\C)}^0(S) \widetilde \times \Spin(S) \\
&=\left(  \RR_{\SL(\C)}(S) \times \Spin(S) \right) / H^1(S; \Z_2)
\end{align*}

An element of $r\in \RS(S)$ associates a well-defined trace to each \emph{framed} knot $K \subset S\times[0,1]$. Indeed, if we represent $r$ by a pair $(\widehat r, \sigma)$ consisting of a group homomorphism $\widehat r\col \pi_1(S) \to \SL(\C)$ and of a spin structure $\sigma \in \Spin$, we can consider the trace
$$
\rho_r(K) = - (-1)^{\sigma(K)} \Tr\, \widehat r(K),
$$
where $\sigma(K) \in \Z_2$ is the monodromy of the framing of $K$ with respect to the spin structure on $S\times[0,1]$ defined by $\sigma$. This clearly depends only on the class $r$ of $(\widehat r, \sigma)$ in $\RS(S)$. 

The isomorphism $\mathcal S^{-1}(S) \to \mathcal S^1(S)$ associated to $\sigma$  by Barrett \cite{Barr} similarly involves a factor $(-1)^{\sigma(K)}$. Combining \cite{Barr} with Theorem~\ref{thm:A=-1} immediately gives:

\begin{thm}
\label{thm:A=1}
An element $r\in \RS(S)$ determines a non-trivial algebra homomorphism
$$
\rho_r\col \mathcal S^1(S) \to \C
$$ 
associating to a connected skein $[K]\in \mathcal S^1(S)$ the trace $\rho_r(K)$ defined above.

Conversely, every non-trivial homomorphism $\rho_r\col \mathcal S^1(S) \to \C$ is associated in this way to an $r\in \RS(S)$. \qed
\end{thm}

Note that the spaces $\RS(S)$ and $\RR_{\SL(\C)}$ are very similar, in the sense that they both are coverings of the component $\RR_{\PSL(\C)}^0$ of the character variety $\RR_{\PSL(\C)}$, with fiber $\Spin(S)\cong H^1(S;\Z_2)$.

We will refer to the elements of $\RS(S)$ as \emph{spinned homomorphisms} from $\pi_1(S)$ to $\PSL(\C)$, as they consist of a group homomorphism $r\col \pi_1(S) \to \PSL(\C)$ together with some additional spin information specifying how to lift $r$ to $\SL(\C)$. 

Hyperbolic geometers are actually quite familiar with spinned homomorphisms. Indeed, suppose that we are given a hyperbolic metric on the thickened surface $S\times (0,1)$, not necessarily complete. The holonomy of this metric gives a homomorphism $r\col \pi_1(S) \to \PSL(\C)$. What is less well-known is that this $r$ comes with additional spin information. More precisely, the choice of a spin structure $\sigma\in \Spin(S)$ specifies a preferred lift of $r$ to $\widehat r \col \pi_1(S) \to \SL(\C)$, and a different choice for the spin structure $\sigma$ leaves the class of $(\widehat r, \sigma)$ in $\RS(S)$ unchanged; see for instance \cite[Proposition~10]{BonWon1}. This proves:

\begin{prop}
A hyperbolic metric on $S\times(0,1)$ uniquely determines a spinned homomorphism $r\in \RS(S)$. \qed
\end{prop}

Theorems~\ref{thm:A=-1} and \ref{thm:A=1} enable one to consider the skein algebra $\SSS(S)$ as a quantization of the character varieties $\RR_{\SL(\C)}(S)$ and $\RS(S)$.  See \cite{Tur1, BFK1, BFK3}.

\section{Small surfaces}
\label{sect:SmallSurfaces}

We now return to our Revised Goal for arbitrary values of $A$. 
The irreducible representations of the skein algebra $\SSS(S)$ have been completely classified for a few surfaces with relatively simple topological type. 

\subsection{The sphere with $\leq 3$ punctures} When $S$ is a sphere with at most three punctures, the skein algebra $\SSS(S)$  is commutative for all values of $A$. This is because any two simple closed curves on $S$ are disjoint after isotopy. By commutativity,  every irreducible representation of $\SSS(S)\to \End(V)$ then has dimension 1, so that canonically  $\End(V)\cong \C$. 

 An immediate extension of Kauffman's uniqueness result for the Kauffman bracket (Theorem~\ref{thm:Kauffman}) then gives the following result. See for instance \cite{Lick}. 

\begin{thm} 
\label{thm:PunctSph}
Let $S$ be a sphere with at most three punctures, the skein algebra $\SSS(S)$ is commutative and:
\begin{enumerate}
\item  If $S$ is the sphere or the disk, then $\SSS(\C) \cong \C$; as a consequence, it admits a unique irreducible representation.
\item If $S$ is the annulus, then $\SSS(S) \cong \C[X]$ where $X$ is represented by a simple closed curve going around the annulus, and two irreducible representations $\SSS(S)\to \End(V)\cong \C$ are isomorphic if and only if they assign the same number  $x\in \C$ to the generator $X$. 
\item If $S$ is the three-puncture sphere, then $\SSS(S) \cong \C[X,Y,Z]$ where the generators $X$, $Y$, $Z$ are represented by simple closed curves parallel to the boundary components; as a consequence,  two irreducible representations $\SSS(S)\to \End(V)\cong \C$ are isomorphic if and only if they assign the same numbers  $x$, $y$, $z\in \C$ to the generators $X$, $Y$, $Z$. \qed
\end{enumerate}
\end{thm}

\subsection{The torus with $\leq 1$ puncture}

The algebraic structure of the skein algebra $\SSS(S)$ is much more interesting when $S$ is the torus with 0 or 1 puncture, although still manageable. 

\begin{thm} [\cite{BullPrz}]
\label{thm:PunctTor}
$ $
\begin{enumerate}
\item When $S$ is the one-puncture torus, the skein algebra $\SSS(S)$ admits a presentation with three generators $X_1$, $X_2$, $X_3$ and the three relations $$AX_iX_{i+1} - A^{-1}X_{i+1}X_i = (A^2 - A^{-2}) X_{i+2}$$ for $i=1$, $2$, $3$. 

\item When $S$ is the torus (with no puncture), the skein algebra is isomorphic to the quotient of the above algebra by the central element
$$ A^2X^2_1+ A^{-2}X_2^2+ A^2X^2_3- AX_1X_2X_3 -2A^2 - 2A^{-2}.$$
\vskip -\belowdisplayskip
\vskip -\baselineskip
\qed
\end{enumerate}
\end{thm}

Bullock and Przytycki \cite{BullPrz} give similar presentations for the skein algebras of the 4--puncture sphere and the 2-puncture torus.

It turns out that the skein algebra of the one-puncture torus, as described by Theorem~\ref{thm:PunctTor}(1), is isomorphic to a certain quantum deformation $\mathrm U_q'(\mathrm{so}_3)$ of $\mathrm{so}_3$ introduced in \cite{GavKli}. In particular, its representation theory has been analyzed in \cite{HavKP, HavPost}. The irreducible representations of $\SSS(S)$ in this case fall into two general categories, each subdivided into a few subcases. The first category arises for all values of $A$, whereas the second category is restricted to the case where $A$ is a root of unity. The representation theory for the unpunctured torus easily follows from the one-puncture case by Theorem~\ref{thm:PunctTor}(2).

In the rest of the article, we will consider the case of a general surface $S$, when $A$ is a root of unity. The representations of the skein algebra that we will encounter are very similar to the cyclic representations of $\mathrm U_q'(\mathrm{so}_3)$ that occur in \cite{HavPost}. 

\section{Invariants of invariants}
\label{sect:Invariants}

Recall that our (optimistic) goal is to classify all irreducible representations of the skein algebra $\SSS(S)$. 
Usually, the way one classifies a mathematical object is by extracting invariants, and by showing that these determine the object up to isomorphism. In our case, we need invariants of representations of the skein algebra, namely invariants of Kauffman brackets, namely invariants of  link invariants!

We will need $A$ to be a root of unity. Consequently, we assume that $A$ is a primitive $N$--root of unity, namely that $A^N=1$ and that $N$ is minimum for this property. In addition, we require that $N$ is odd. 

\subsection{Chebyshev polynomials}

Recall that the (normalized) \emph{$n$--th Chebyshev polynomial} (of the first type) is the polynomial $T_n(x)$ defined by the property that
$$
\Tr\, M^n = T_n(\Tr\, M) \text{ for every } M \in \SL(\C).
$$
In particular,
$T_0(x) = 2 $,
$T_1(x) = x $,
$T_2(x) = x^2-2 $,
$ T_3(x) = x^3 -3x$, 
 and $T_n(x)$ is more generally defined by the induction relation
 $$
  T_{n+1}(x) = xT_n(x) - T_{n-1}(x). 
$$

\begin{lem}
\label{lem:ChebCentral}
Suppose that $A$ is a primitive $N$--root of unity with $N$ odd. Let $K$ be a framed knot in $ S\times[0,1]$ that is represented by a simple closed curve in $S$, with vertical framing. Then, the evaluation $T_N([K])\in \SSS(S)$ of the $N$--th Chebyshev polynomial is  is central in the skein algebra $\SSS(S)$.
\end{lem}
\begin{proof} [Sketch of proof]
One easily reduces the problem to the case of the one-puncture torus. 
The result then follows from the combination of \cite{BullPrz} and of the construction of central elements of $\mathrm U'_q(\mathrm{so}_3)$ in \cite{HavPost}. The more geometrically inclined reader may find the brute force computation of \cite{HavPost} (where the connection with Chebyshev polynomials is not explicit, but was later observed by V.~Fock) somewhat frustrating, and may prefer the more geometric argument provided by the Product-to-Sum Formula of \cite[Theorem~4.1]{FroGel}. 
\end{proof}

\subsection{The classical shadow}

Suppose that we are given an irreducible representation $\rho \col \SSS(S) \to \End(V)$.  If $K \subset S \times [0,1]$ is represented by a simple closed curve in $S$, with vertical framing, Lemma~\ref{lem:ChebCentral} asserts that $T_N([K])$ is central in $\SSS(S)$. As a consequence, Schur's lemma implies that the image of $T_N([K])$ under the irreducible representation $\rho$ is a multiple of the identity, and there exists a number $\kappa_\rho(K)\in \C$ such that 
$$
\rho\bigl( T_n ([K])\bigr) = \kappa_\rho (K) \Id_V.
$$

\begin{lem}
\label{lem:ChebA=1}
The above map $K\mapsto \kappa(K)$ induces an algebra homomorphism
$$
\kappa_\rho\col  \mathcal S^{+1}(\Sigma)  \to \C.
$$
\end{lem}
\begin{proof}
This essentially is equivalent to the property that $\kappa$ satisfies the skein relation for $A=+1$. Again, one  first reduces the problem to the case of the one-puncture torus (but this here requires using a formula such as \cite[Lemma~14.1]{Lick} and the hypothesis that $N$ is odd), and then apply the Product-to Sum Formula of \cite[Theorem~4.1]{FroGel}. 
\end{proof}

If the homomorphism $\kappa$ of Lemma~\ref{lem:ChebA=1} is non-trivial,  Theorem~\ref{thm:A=1} associates to it an element $r_\rho \in \RS(\Sigma)$. By definition, this spinned homomorphism  $r_\rho \col \pi_1(\Sigma) \to \PSL(\C)$ is the \emph{classical shadow} of the representation
$$\rho\col \SSS(\Sigma) \to \End(V).$$
(The classical shadow is undefined if $\kappa_\rho$ is trivial.)

\subsection{Puncture invariants} When $S$ has non-empty boundary, the skein algebra $\SSS(S)$ contains central elements that are more obvious than those defined using Chebyshev polynomials. Indeed, if $P_i$ is a curve parallel to the $i$--th boundary component of $S$ and is endowed with the vertical framing, the skein $[P_i]\in \SSS(S)$  is clearly central in $\SSS(S)$. 

As before, an irreducible representation  $\rho \col \SSS(S) \to \End(V)$ provides a number $p_i\in \C$ such that
$$\rho( [P_i]) = p_i \Id_V.$$
This number $p_i\in \C$ is the \emph{$i$--th puncture invariant} of the irreducible representation~$\rho$. 

There clearly should be a relationship between these puncture invariants and the classical shadow $r_\rho$, involving the Chebyshev polynomial $T_N$. 

\begin{lem}
If $r_\rho \in \RS(S)$ is the classical shadow of the irreducible representation $\rho \col \SSS(S) \to \End(V)$,  
$$T_N(p_i) = \kappa_\rho([P_i]) =- \Tr\, r_\rho(P_i) $$
where $\Tr\, r_\rho(P_i) = (-1)^{\sigma(P_i)} \Tr\, \widehat r_\rho(P_i)$ is defined as in {\upshape\S \ref{sect:A=+1}}. 
\qed
\end{lem}

Namely, the puncture invariant $p_i$ belongs to the finite set $T_N^{-1} \bigl( - \Tr\, r_\rho(P_i)  \bigr)$, and is therefore determined by the classical shadow up to finitely many choices.

\section{The conjectural classification of irreducible representations}
\label{sect:ConjClass}

Remember that we are restricting attention to the case where the number $A$ intervening in the definition of the skein algebra $\SSS(S)$ is a primitive $N$--root of unity, with $N$ odd. 

We then associated to each irreducible representation $\rho \col \SSS(S) \to \End(V)$ a spinned homomorphism $r_\rho \in \RS(S)$ (unless the algebra homomorphism $\kappa_\rho$ of Lemma~\ref{lem:ChebA=1} is trivial) and puncture invariants $p_i \in T_N^{-1} \bigl( - \Tr\, r_\rho (P_i)  \bigr)$. 

\begin{conj}
\label{conj:Classif}
Suppose that we are given:
\begin{enumerate}
\item a spinned homomorphism $r \col \pi_1(\Sigma) \to \PSL(\C)$;
\item for each boundary component of $S$, a number $p_i \in T_N^{-1} \bigl( - \Tr\, r(P_i)  \bigr)$
\end{enumerate}
Then, up to isomorphism,  there is a unique irreducible representation $\rho\col \SSS(\Sigma) \to \End(V)$ whose invariants are $r$ and the $p_i$.
\end{conj}

Note that this conjecture leaves aside the irreducible representations for which the homomorphism $\kappa_\rho\col \mathcal S^1(S) \to \C$ provided by Lemma~\ref{lem:ChebA=1} is trivial. We conjecture that these fit in a larger picture, valid for all values of $A$. 

More precisely, recall that the (finite-dimensional) irreducible representations of the quantum group $\mathrm U_q(\mathrm{sl}_2)$ fall into two wide categories: the \emph{generic representations}, which occur for all values of $q$, and the \emph{cyclic representations} restricted to the case where $q$ is a root of unity. The generic representations are essentially rigid, depending only on the dimension and on a finite amount of additional information, but there are whole moduli spaces of cyclic representations, which depend on finitely many continuous choices of parameters. See for instance \cite[\S VI.5]{Kass}. 

Each of these (types of) representations of $\mathrm U_q(\mathrm{sl}_2)$ should induce a representation of the skein algebra $\SSS(S)$ when $A=q^{\frac12}$, by an extension of the framework of \cite{BonWon1} discussed below (see also \cite{BFK2} for more ideas), and we conjecture that all irreducible representations of $\SSS(S)$ should be obtained in this way.

In this framework, the generic representations of $\mathrm U_q(\mathrm{sl}_2)$ should essentially give unique representations of $\SSS(S)$. 

When $A$ is a root of unity, the cyclic representations  of $\mathrm U_q(\mathrm{sl}_2)$ should give the representations of $\SSS(S)$ discussed in Conjecture~\ref{conj:Classif} (with minor variations when $A$ is a primitive $N$--root of unity with $N$ even). The classical shadow is here a way to keep track of the moduli parameters describing the cyclic representations appearing in the construction.

We conjecture that, conversely, every irreducible representation of the skein algebra is obtained in this way, namely is associated by such a construction to irreducible representations of the quantum group $\mathrm U_q(\mathrm{sl}_2)$.

\medskip
This conjectural picture will clearly need to be adjusted as our understanding progresses. At this point, however, we are able to prove the existence part of Conjecture~\ref{conj:Classif}.

\begin{thm}
\label{thm:MainThm}
Suppose that $A$ is a primitive $N$--root of unity with $N$ odd, and that we are given:
\begin{enumerate}
\item a spinned homomorphism $r \col \pi_1(\Sigma) \to \PSL(\C)$;
\item for each boundary component of $S$, a number $p_i \in T_N^{-1} \bigl( - \Tr\, r(P_i)  \bigr)$
\end{enumerate}
Then there exists an irreducible representation $\rho\col \SSS(\Sigma) \to \End(V)$ whose classical shadow is $r\in \RS(S)$ and whose puncture invariants are equal to the $p_i$. 
\end{thm}

The rest of this article is devoted to a discussion of the proof of Theorem~\ref{thm:MainThm}. Details will appear elsewhere \cite{BonWon2}. 

\section{Triangulations}
\label{sect:Triang}

For simplicity, we will focus attention to the case where the surface is closed, namely without punctures. This also turns out to be the harder case. 

We follow a strategy proposed by other authors \cite[Chap.~3]{Pal}, and drill holes in $S$. More precisely, let $S_0$ be a surface obtained by removing finitely many points $v_1$, $v_2$, \dots, $v_p$ from $S$. 

Choose an ideal triangulation $T$ for $S_0$, namely a triangulation of the closed surface $S$ whose vertices are exactly the punctures $v_1$, $v_2$, \dots, $v_p$. We require that the end points of each edge are distinct, which is easily achieved by removing additional punctures if necessary.

We now associate an algebraic object to this triangulation.  Its representation theory is relatively simple, and will help us understand that of the skein algebra. 

\subsection{The Chekhov-Fock algebra of a train track} 

In each triangle of the triangulation $T$, join the midpoints of the edges by three arcs with disjoint interiors, meeting orthogonally the edges as in Figure~\ref{fig:Triangle}. 

\begin{figure}[htbp]
\includegraphics{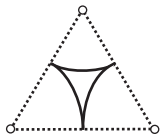}

\caption{A triangle and its train track}
\label{fig:Triangle}
\end{figure}

These arcs fit together to form a train track $\tau$ on the surface $S_0$, as illustrated in  Figure~\ref{fig:TriangTrainTrack}.

\begin{figure}[htbp]

\includegraphics{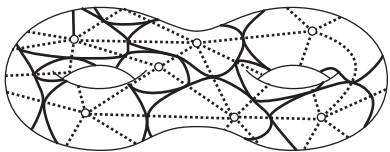}
\caption{A triangulation and its associated train track}
\label{fig:TriangTrainTrack}
\end{figure}

To keep in line with the themes of the conference, this train track $\tau$ is traditionally  used to specify the 1--dimensional submanifolds of $S_0$ that are normal with respect to the triangulation. More precisely, these normal 1--submanifolds are in one-to-one correspondence with systems of  non-negative integer edge weights for $\tau$ satisfying the \emph{Switch Condition} that, at each switch, the weights of the two edges incoming on one side of the switch add up to the sum of the weights of the two edges outgoing on the other side. 

Let $\mathcal W(\tau; \Z)$ be the set of integer weight systems for $\tau$ that satisfy the Switch Condition. Note that we are not any more restricting these weights to be non-negative.  The set $\mathcal W(\tau; \Z)$ is an abelian group under addition, and is easily computed to be isomorphic to $\Z^{6g+3p-6}$. 

The space $\mathcal W(\tau; \Z)$ is equipped with a very natural antisymmetric bilinear form, the \emph{Thurston intersection form}
$$
\omega \col \mathcal W(\tau;\Z) \times \mathcal W(\tau;\Z) \to \Z.
$$

The Thurston form is essentially an algebraic intersection number. To define $\omega$,  perturb $\tau$ to a train track $\tau'$ that is transverse to $\tau$. Each intersection point $x$ of $\tau$ with $\tau'$ comes with a sign: indeed, any choice of local orientation of $\tau$ near $x$ specifies a local orientation of $\tau'$, if the train tracks are sufficiently $\mathrm C^1$--close; reversing the local orientation of $\tau$ reverses the local orientation of $\tau'$, and consequently does not change the algebraic sign of the intersection of $\tau$ with $\tau'$ at $x$. Then, given $\alpha$, $\beta\in \mathcal W(\tau; \Z)$, the number $\omega(\alpha, \beta)$  counts  these intersection signs when the edges of $\tau$ are equipped with the multiplicity $\alpha$ and the edges of $\tau'$ with the multiplicity $\beta$. 

We now define an algebra $\T_\tau^A$ as follows. As a vector space, $\T_\tau^A$ is freely generated by the elements of  $ \mathcal W(\tau;\Z) $; namely, $\T_\tau^A$ consists of all formal linear combinations $\sum_{i=1}^j a_i\alpha_i$ where $a_i\in \C$ and $\alpha_i\in \mathcal W(\tau; \Z)$. If we defined a multiplication on $\T_\tau^A$ by the group law of $ \mathcal W(\tau;\Z) $, namely by the property that $\alpha \cdot \beta = (\alpha+\beta) \in  \mathcal W(\tau;\Z) $ for $\alpha$, $\beta\in \mathcal W(\tau; \Z)$, we would obtain the group algebra $\C[ \mathcal W(\tau;\Z) ]$. Instead we twist this group law by the constant $A$ and by the Thurston form $\omega$. More precisely, we define the multiplication of two elements $\alpha$, $\beta\in \mathcal W(\tau; \Z)$ as 
$$
\alpha \cdot \beta = A^{\frac12\omega(\alpha, \beta)}(\alpha+\beta) 
$$
and linearly extend this multiplication to $\T_\tau^A$.

This algebra $\T_\tau^A$ is the \emph{Chekhov-Fock algebra} of the train track $\tau$. The reader will recognize here one of the many avatars of the \emph{quantum Teichm\"uller space} \cite{Foc, CheFoc1, Kash, BonLiu, Liu1}. If we restrict attention to those edge weight systems $\alpha \in  \mathcal W(\tau;\Z)$ where, at each switch, the weights of the two edges incoming on one side have the same parity, the corresponding subalgebra of $ \mathcal W(\tau;\Z)$ is exactly the Chefock-Fock algebra of the ideal triangulation $T$, as defined in \cite{BonLiu} and for $q=A^{-2}$. Therefore, $\T_\tau^A$ is obtained from the Chekhov-Fock algebra of $T$ by augmenting it with certain square roots of generators. The Chekhov-Fock algebra $\T_\tau^A$ appears in a different presentation  in \cite[\S 2.3]{BonWon1}, where it is called  $\mathcal Z_T^{A^{-\frac12}}$.

\subsection{From the skein algebra to the quantum Teichm\"uller space}

The algebraic structure of the Chekhov-Fock algebra $ \T_\tau^A$ is remarkably simple. A key step in our analysis connects it to the skein algebra $\SSS(S_0)$. 

\begin{thm}[\cite{BonWon1}]
\label{thm:QTraces}
There exists a natural embedding $\SSS(S_0) \to \T_\tau^A$. \qed
\end{thm}

When $A=1$, the homomorphism $\mathcal S^1(S) \to \T_\tau^1$ of Theorem~\ref{thm:QTraces} has a geometric interpretation. Using the technique of pleated surfaces, a system of complex weights associated to the edges of $T$ determines a homomorphism $r\col \pi_1(S_0) \to \PSL(\C)$. Choosing suitably compatible square roots for these complex numbers  even determines a spinned homomorphism $r\in \RS(S_0)$. The homomorphism $\mathcal S^1(S) \to \T_\tau^1$ then associates to a framed link $K\subset S\times [0,1]$ the Laurent polynomial in these square roots that gives $\Tr\,r(K)$ for the associated $r\in \RS(S)$. See \cite[\S 1]{BonWon1}. The general homomorphism $\SSS(S_0) \to \T_\tau^A$ can be seen as a non-commutative deformation of this formula. The challenge is to make it natural, namely well-behaved with respect to changing the ideal triangulation $T$ of $S_0$ (and therefore with respect to the action of the mapping class group of $S_0$). 

Theorem~\ref{thm:QTraces} was conjectured in \cite{CheFoc2}. See also \cite{ChePen1, Hiatt} for partial results in this direction. 

The  Chekhov-Fock algebra $ \T_\tau^A$ is an example of a quantum torus and, as such, its irreducible representations are easy to classify. In fact, when $A$ is a primitive $N$--root of unity, they are classified by elements of $\mathcal W(\tau; \C)$, namely complex edge weight systems, together with choices of $N$--roots for certain complex numbers associated to the boundary components of $S_0$ by the complex edge weights. 

Interpreting these complex edge weights as shear-bend parameters for pleated surfaces provides Proposition~\ref{prop:BonLiu} below. We need to introduce some definitions. 

A \emph{peripheral subgroup} of $\pi_1(S_0)$ is one associated to a puncture, namely generated by a loop $P_i$ going around a puncture of $S_0$. Note that the same puncture defines many peripheral subgroups, each associated to the path $\gamma_i$ joining $P_i$ to the  base point used in the definition of the fundamental group $\pi_1(S)$. 
In particular, $\pi_1(S)$ acts on its peripheral subgroups by conjugation. 

Two peripheral subgroups are \emph{connected by the edge $e$} of the triangulation $S$ if they are respectively defined by loops $P_i$, $P_j$ and paths $\gamma_i$, $\gamma_j$ such that the path $\gamma_i\gamma_j^{-1}$ is homotopic to a path in this edge $e$ by a homotopy keeping path endpoints in  $P_i \cup P_j$.

\begin{prop}[\cite{BonLiu}]
\label{prop:BonLiu}
There is a one-to-one correspondence between isomorphism classes of irreducible representations $\rho\col \T_\tau^A \to \End(V)$ and sets of data consisting of:
\begin{enumerate}
\item[\upshape(1)] a spinned homomorphism $r_0 \col \pi_1(S_0) \to \PSL(\C)$, namely an element $r_0 \in \RS(S_0)$;
\item [\upshape(2)]for each peripheral subgroup $\pi\subset \pi_1(S)$, a line $\xi_\pi \subset \C^2$ such that 
\begin{enumerate}
\item [\upshape(a)] $\xi_\pi$ is respected by the subgroup $r_0(\pi)\subset \SL(\C)$;
\item [\upshape(b)] $\xi_{\gamma\pi\gamma^{-1}} = r_0(\gamma) (\xi_\pi)$ for every $\gamma\in \pi_1(S_0)$;
\item  [\upshape(c)]$\xi_\pi \not = \xi_{\pi'}$ when the parabolic subgroup $\pi$ and $\pi'$ are connected by an edge of the triangulation;
\end{enumerate}
\item[\upshape(3)]  for each puncture, a number $h_i\in \C$ whose power $h_i^N$ is determined by the following property: if $r_0\in \RS(S)$ is represented by $r_0'\col \pi_1(S) \to \SL(\C)$ and by a spin structure $\sigma\in \Spin(S)$, if $P_i\in \pi_1(S)$ is represented by a small loop going counterclockwise around the $i$--th puncture, and if $\pi_1$ is the peripheral subgroup generated by $P_i$, then $r_0'$ acts by multiplication by $(-1)^{\sigma(P_i)}h_i^N$ on the line $\xi_{\pi_i}\subset \C^2$. 
\end{enumerate}
\end{prop}

Proposition~\ref{prop:BonLiu} is an easy extension of some arguments of \cite{BonLiu} to the square root context considered here. 
The harder part of \cite{BonLiu} is to prove that Proposition~\ref{prop:BonLiu} is well-behaved under change of the ideal triangulation $T$, but we will not need this property here.

\subsection{Constructing representations of the skein algebra} Let us repeat the statement we intend to prove.

We are restricting attention to the case where the surface is closed, so that we do not have to worry about puncture invariants; this is only for the sake of exposition, and the general case is very similar. We also assume that $A$ is a primitive $N$--root of unity with $N$ odd, a critical hypothesis at this point. We are given a spinned homomorphism $r\in \RS(S)$, and we want to construct an irreducible representation $\rho\col \pi_1(S) \to \End(V)$ with classical shadow $r$. 

Choose a triangulation $T$ for $S$. We require that the end points of each edge are distinct. Let  $v_1$, $v_2$, \dots, $v_p$ be the vertices of this triangulation, and consider the punctured surface $S_0=S-\{v_1, v_2, \dots, v_p\}$. The spinned homomorphism $r\in \RS(S)$ gives by restriction a spinned homomorphism $r_0 \in \RS(S_0)$.

To apply Proposition~\ref{prop:BonLiu} to construct a representation $\T_\tau^A \to \End (V)$, we need to choose suitable invariant lines $\xi_\pi \subset \C^2$ and numbers $h_i \in \C$. In our case, the restriction of $r_0$ to each peripheral subgroup is trivial. This makes the choice of the lines $\xi_\pi$ particularly easy, since Condition~(2a) of Proposition~\ref{prop:BonLiu} is automatically satisfied. In fact, there is a whole $2p$--dimensional family of possible choices for the lines $\xi_\pi$. Similarly, the data (3) is simplified by this property of $r_0$, and just means that $h_i$ is an $N$--root of unity. 

Having chosen lines $\xi_\pi$ and weights $h_i$ as above, Proposition~\ref{prop:BonLiu} provides a representation $\T_\tau^A \to \End (V)$. Combining it with the homomorphism $\SSS(S_0) \to  \T_\tau^A$ of Theorem~\ref{thm:QTraces}, we now have a representation $\rho_0\col \SSS(S_0) \to \End(V)$. This is beginning to look like the statement we are trying to prove for Theorem~\ref{thm:QTraces}, but many things need to be checked:

\begin{enumerate}
\item [\upshape(1)] Is the classical shadow of $\rho_0$ equal to $r_0$?
\item [\upshape(3)] Is $\rho_0$ irreducible?
\item [\upshape(3)] Does $\rho_0\col \SSS(S_0) \to \End(V)$ induce a representation $\rho\col \SSS(S) \to \End(V)$ of the skein algebra of $S$? Namely is it true that $\rho_0([K]) = \rho_0([K']) $  when the framed links $K$, $K' \subset S_0 \times [0,1]$ are isotopic in $S\times[0,1]$, by an isotopy that is allowed to cross the punctures? 
\item [\upshape(4)] Is $\rho$ independent of the choices that we have made, and in particular of  the lines $\xi_\pi \subset \C^2$ and the triangulation of $S$?
\end{enumerate}

These issues are then dealt with in the following way.
\begin{enumerate}

\item The answer to (1) is yes, by the occurrence of many ``miraculous cancelations'' in the computation of Chebyshev polynomials $T_N(\rho([K]))\in \End(V)$. 

\item The answer to (2) is a definite no. However, we can bypass this difficulty by restricting $\rho$ to an irreducible component.

\item After this irreducible reduction, the answer to (3) then remarkably becomes yes,  provided that we chose all $h_i=A^{-1}$ and that we restricted $\rho$ to an appropriate irreducible component.

\item It then turns our that the answer to (4) is also positive.

\end{enumerate}

Proving these statements completes the proof of Theorem~\ref{thm:MainThm}. Details, which are occasionally very elaborate,  will appear elsewhere \cite{BonWon2}.

\end{document}